\documentclass[12pt]{amsart}
\usepackage{geometry} % see geometry.pdf on how to lay out the page. There's lots.
\usepackage{amssymb}
\usepackage{amsthm}
\usepackage{amsmath}

\newcommand{\R}{\mathbb{R}}
\newcommand{\C}{\mathbb{C}}

\title{Lifting tensors from orbifold quotients}
\author{Ricardo A. E. Mendes}

\newtheorem{theorem}{Theorem}[section]
\newtheorem*{theorem2}{Theorem}
\newtheorem{theorem3}{Theorem}
\newtheorem{proposition}{Proposition}
\theoremstyle{remark}

\begin{document}

\maketitle
%\tableofcontents

\section{Introduction and Main Results}

Let $M$ be a Riemannian manifold and $G$ a Lie group acting on $M$ properly by isometries. The orbit space $M/G$ is naturally a metric space with curvature (locally) bounded from below, more precisely an \emph{Alexandrov space} --- see \cite{Grove02}.

A natural question by K. Grove is to describe the metrics on the space $M/G$ which are induced by some smooth $G$-invariant metric on $M$.

We will consider the case where $M/G$ lies in the more restrictive class of \emph{Riemannian orbifolds}, defined as in \cite{LytchakThorbergsson10}: $M/G$ has a cover $\mathcal{U}$ by open sets, each isometric to a quotient $N/\Gamma$ of a Riemannian manifold $N$ by a finite group of isometries $\Gamma$. These ``charts'' satisfy certain compatibility relations, and in fact the metric structure on $M/G$ determines a unique orbifold structure on $M/G$, as defined in \cite{BridsonHaefliger}. The familiar tools from differential geometry often apply to orbifolds. For example smooth tensors (in particular smooth functions) are defined first locally on each $U\in\mathcal{U}$ with $U= N/\Gamma$ to be the smooth $\Gamma$-invariant tensors on $N$. Globally on $M/G$ one defines smooth tensors to be collections of smooth tensors on each $U\in\mathcal{U}$ which agree on intersections.

The main result in the present paper is a solution to K.Grove's question in case the quotient is an orbifold:
\begin{theorem}
\label{mainthm}
Let $M$ be a Riemannian manifold, with the proper isometric action of a Lie group $G$. Assume that the quotient $M/G$ is (isometric to) a Riemannian orbifold. Then any smooth symmetric $2$-tensor $g$ on $M/G$ is induced by some smooth $G$-invariant symmetric $2$-tensor $\tilde{g}$ on $M$. Moreover if $g$ is a metric, then $\tilde{g}$ can be chosen to be a metric as well.
\end{theorem}

Theorem \ref{mainthm} is a consequence of an extension result
for symmetric $2$-tensors invariant under a \emph{polar} action. An isometric action on $M$ by a Lie group $G$ is called polar if there exists an immersed submanifold
 $\Sigma\subset M$ meeting all $G$-orbits orthogonally. Such a
 submanifold $\Sigma$ is called a section, and comes with a
 natural action by a discrete group of isometries, called its
 \emph{generalized Weyl group}.

In \cite{LytchakThorbergsson10} (Theorem 1.1) A.~Lytchak and G. Thorbergsson characterized isometric actions with orbifold quotients: they are precisely the ones that are \emph{infinitesimally polar}, that is, all slice representations are polar. Therefore Theorem \ref{mainthm} is a consequence of the following statement about polar actions:

\begin{theorem}
\label{extendingsym2tensors}
 Let $M$ be a polar $G$-manifold with section $\Sigma$, and $W(\Sigma)$ the generalized Weyl group associated to $\Sigma$.
Let $\sigma$ be a $W(\Sigma)$-invariant symmetric $2$-tensor on the section $\Sigma$.

Then there exists a smooth $G$-invariant symmetric $2$-tensor $\tilde{\sigma}$ on $M$ whose pull-back to $\Sigma$ equals $\sigma$.

Moreover, if $\sigma$ is a metric, then we can choose $\tilde{\sigma}$ to be a metric with respect to which the action of $G$ remains polar with the same sections.
\end{theorem}

The general idea of the proof is to combine Chevalley's Restriction Theorem with a description of 
$W(\Sigma)$-invariant symmetric $2$-tensors in terms of invariant functions. The latter follows from an algebraic result:

\begin{theorem}[Hessian Theorem]
\label{hessthm}
Let $W\subseteq O(V)$ be a finite reflection group, where $V$ is a Euclidean vector space of dimension $n$. Denote by $\R[V]^W$  the $\R$-algebra of $W$-invariant real-valued polynomials on $V$. 

Consider the space of $W$-invariant symmetric $2$-tensors on $V$ with polynomial coefficients, and denote it by $\R[V,\mathrm{Sym}^2V^*]^W$.

Then there are  $W$-invariant homogeneous polynomials $Q_1, \ldots Q_l$ whose Hessians form a basis for $\R[V,\mathrm{Sym}^2V^*]^W$ as a module over $\R[V]^W$. Here $l=\dim(\mathrm{Sym}^2 V ^*$)$=(n^2+n)/2$.
\end{theorem}

If one chooses a set of homogeneous generators for $\R[V]^W$, then each $Q_i\in\R[V]^W$ in the statement of the Hessian theorem can be chosen to be either a generator, or a product of two generators. For example when $W$ is the dihedral group of order $2n$ acting on $\R^2=\C$, it is well known that $\rho_1=|z|^2$, $\rho_2=\mathrm{Re}(z^n)$ generate the polynomial invariants, and in this case we can take $\{Q_1, Q_2, Q_3\}= \{\rho_1, \rho_2, \rho_1^2 \}$.

Here is a brief outline of the proof of theorem \ref{extendingsym2tensors}, using the Hessian Theorem. We first use the Slice theorem to reduce to the case where $M$ is a Euclidean vector space and $G$ acts linearly, that is, to the case of \emph{polar representations}. Then we reduce further to the case where $G$ is connected. Polar representations of connected groups were shown by Dadok \cite{Dadok85} to be orbit-equivalent to an isotropy representation of a symmetric space. This implies that the generalized Weyl group $W= W(\Sigma)$ is generated by reflections. Then the Hessian theorem, together with an argument involving the Malgrange Division theorem, implies that there are  $W$-invariant polynomials $Q_i$ and $a_i\in C^\infty(V)^W$ such that the given symmetric $2$-tensor $\sigma$ can be written as $ \sigma = \sum_i a_i \mathrm{Hess}(Q_i)$. Finally the Chevalley Restriction theorem says that $a_i,Q_i$ can be extended to $G$-invariant functions $\tilde{a_i},\tilde{Q_i}$ on $M$, so that we may define $ \tilde{\sigma} = \sum_i \tilde{a_i} \mathrm{Hess}(\tilde{Q_i})$.

For the proof of the Hessian theorem, we first reduce to the the case where $W$ is irreducible, and then use the classification of irreducible reflection groups by type. For each type, we compute the Poincar\'e series $P(t)$ of the graded vector space of polynomial symmetric $2$-tensors through Molien's formula. Then we find explicit sets $\{Q_i\}$ of homogeneous invariants whose degrees match $P(t)$, and show by direct computation of a determinant that the set  $\{\mathrm{Hess}Q_i\}$ is linearly independent. For the exceptional types a computer is used.

We note that theorem \ref{extendingsym2tensors} is analogous to Michor's Basic Forms Theorem about polar manifolds --- see \cite{Michor96} and \cite{Michor97}. Indeed, Michor's theorem states that for a polar $G$-manifold $M$ with section $\Sigma$, every smooth $W(\Sigma)$-invariant $p$-form on $\Sigma$ can be extended uniquely to a smooth $G$-invariant $p$-form on $M$, vanishing when contracted to vectors tangent to the $G$-orbits.
 
Our Hessian Theorem is an analogue of a theorem by Solomon about finite reflection groups --- see \cite{Kane} section 22. It says in particular that if $W\subset O(\Sigma)$ is a reflection group in a Euclidean vector space $\Sigma$, every $W$-invariant $p$-form on $\Sigma$ with polynomial coefficients can be written as a sum of terms of the form $Q_0\cdot dQ_1\wedge\ldots\wedge dQ_p$, where each $Q_i$ is a $W$-invariant polynomial.

One naturally wonders if the Hessian Theorem and Solomon's Theorem generalize to other types of tensors beyond Sym$^2$ and $\Lambda^p$, for example the higher symmetric powers. Such generalizations would imply the corresponding tensor \emph{extension} theorems for polar manifolds and tensor \emph{lifting} theorems for all $G$-manifolds with orbifold quotients.

The present paper is organized as follows: In section \ref{sectionpolar} we recall definitions and some facts about polar actions and prove Theorems \ref{mainthm} and \ref{extendingsym2tensors} using the Hessian Theorem. Section \ref{sectionhessthm} concerns the Hessian Theorem. It starts with general remarks, including a description of the framework common to the proofs in all types, followed by the actual proofs in each type: dihedral groups, classical groups, and finally exceptional groups --- see Theorems \ref{dihedral}, \ref{classical} and \ref{exceptional}, respectively.

Acknowledgements: This work was completed as part of my PhD, and I would like to thank my advisor W. Ziller for the long-term support. I would also like to thank A. Lytchak, H. Wilf, C. Krattenthaler, and P. Levande for useful communication.

\section{Extending metrics in polar manifolds}
\label{sectionpolar}

In the present section we first show how Theorem \ref{mainthm} follows from Theorem \ref{extendingsym2tensors}, and then we prove Theorem \ref{extendingsym2tensors} using the Hessian Theorem.

\begin{proof}[Proof of Theorem \ref{mainthm}]
Let $x\in M$, $K=G_x$ the isotropy and $V=(T_xGx)^\perp$ the slice. By Theorem 1.1 in \cite{LytchakThorbergsson10} the fact that $M/G$ is an orbifold implies that the action of $K$ on $V$ is polar. If $\Sigma\subset V$ is a section, a neighbourhood of $Gx$ in $M/G$ can then be identified with $\Sigma/W(\Sigma)$.

Because the action of $K$ on $V$ is polar with section $\Sigma$, the action of $G$ on $G\times_K V$ is polar with section the image $\bar{\Sigma}$ of $\{1\}\times \Sigma$ in $G\times_K V$, if an appropriate metric is chosen for $G\times_K V$. An example of such a metric is obtained as the quotient of the product metric of a left $G$-invariant, right $K$-invariant metric on $G$ with the Euclidean $K$-invariant metric on $V$.

Since $\bar{\Sigma}$ can be identified with $\Sigma$, and its generalized Weyl group is $W(\Sigma)$, it follows from Theorem \ref{extendingsym2tensors} that any smooth symmetric $2$-tensor on a small neighbourhood of $Gx$ in $M/G$ is induced by some smooth $G$-invariant tensor on $G\times_K V$, and hence also on a neighbourhood of the orbit $Gx$ in $M$ by the Slice Theorem. Finally with a $G$-invariant partition of unity we can define such a tensor on all of $M$.
\end{proof}

Now we turn to the proof of Theorem \ref{extendingsym2tensors}. We start by recalling the definition and a few facts about polar actions (see \cite{PalaisTerng87} for more information).

Let $M$ be a Riemannian manifold, and $G$ a Lie group acting on $M$ by isometries. The action of $G$ on $M$ is called \emph{polar}, and $M$ is called a polar manifold, if there exists an isometric immersion $i: \Sigma \to M$, called a \emph{section}, which meets all $G$-orbits and always orthogonally. Moreover we require that $\Sigma$ has no subcover section, that is, the immersion $i: \Sigma\to M$ does not factor as a covering map followed by an immersion $\Sigma \to \Sigma'\to M$. We note that $\Sigma$ is totally geodesic in $M$. The simplest example is the action of $SO(2)$ on $\R^2$ by rotations, where any straight line through the origin is a section. Also note that in \cite{PalaisTerng87} sections are required to be closed and embedded --- see \cite{GroveZiller} for a treatment of immersed sections.

To each section $\Sigma$ is associated a discrete group, called the generalized Weyl group, defined as the quotient $ W=W(\Sigma)=N(i\Sigma)/Z(i\Sigma) $ of the normalizer $N(i\Sigma)=\{g\in G\  |\ gi\Sigma=i\Sigma\}$ by the centralizer $Z(i\Sigma)=\{g\in G\ |\ gx=x\ \ \forall x\in i\Sigma\}$ of the image of $\Sigma$ in $M$. The natural action of $W$ on $i\Sigma$ lifts uniquely to an action on $\Sigma$, such that $i:\Sigma\to i\Sigma$ is $W$-equivariant. The $W$-orbits in $\Sigma$ are exactly the inverse images by $i$ of the $G$-orbits in $M$, and in fact the immersion $i:\Sigma\to M$ induces an isometry of quotient spaces $\Sigma /W\to M/G$. A less elementary fact is that the quotients also have the same smooth structure, in the following sense:
\begin{theorem2}[Chevalley Restriction Theorem]
%\label{Chevalley}
Pull-back to the section $\Sigma$ induces an isomorphism $i^*:C^\infty(M)^G\to C^\infty(\Sigma)^W$ between the algebras of smooth invariant functions.
\end{theorem2}

Here is the linear version of Theorem \ref{extendingsym2tensors}. The proof makes use of the Hessian theorem.

\begin{proposition}
\label{linearversion}
Let $V$ be a polar $K$-representation, where $K$ is a compact Lie group, not necessarily connected, with section $\Sigma\subset V$ and generalized Weyl group $W$. Denote by $i:\Sigma\to V$ the inclusion.

Let $\sigma\in C^\infty(\Sigma,\mathrm{Sym}^2\Sigma^*)^W$ be a smooth $W$-equivariant map from $\Sigma$ to $\mathrm{Sym}^2\Sigma^*$. Then there exists
$\tilde{\sigma}\in C^\infty(V,\mathrm{Sym}^2 V^*)^K$ such that $i^*\tilde{\sigma}=\sigma$ and $\tilde{\sigma}(X,Y)=0$ for $X$ horizontal and $Y$ vertical. 

Moreover, if $\sigma$ is positive definite at $0$, then so is $\tilde{\sigma}$.
\end{proposition}

Recall that vertical means tangent to the orbit, and horizontal means normal to the orbit.

\begin{proof}

Let $K_0$ denote the connected component of $K$ containing the identity, and $N(\Sigma)_0$ and $Z(\Sigma)_0$ the normalizer and centralizer of $\Sigma$ in $K_0$, which equal $N(\Sigma)_0=N(\Sigma)\cap K_0$ and $Z(\Sigma)_0=Z(\Sigma)\cap K_0$.

From Dadok's classification it follows that the representation of $K_0$ on $V$ is equivalent to the isotropy representation of a symmetric space --- see \cite{Dadok85}. In particular $W(\Sigma)_0=N(\Sigma)_0/Z(\Sigma)_0$ is a Weyl group, that is, a crystallographic reflection group. Therefore, by the Hessian theorem, there are homogeneous $W_0$-invariant polynomials $Q_1,\ldots Q_l$ whose Hessians generate the module  $\R[\Sigma,\mathrm{Sym}^2\Sigma^*]^{W_0}$ of all $W_0$-invariant symmetric $2$-tensors on $\Sigma$ with polynomial coefficients over the algebra of invariants $\R[V]^{W_0}$.

By an argument involving Malgrange's Division Theorem, the fact that $\{\mathrm{Hess}(Q_1)$,  $\ldots,\mathrm{Hess}(Q_l)\}$ generate $\R[\Sigma,\mathrm{Sym}^2\Sigma^*]^{W_0}$ over $\R[V]^{W_0}$ implies the corresponding statement in the smooth category. Namely, $\{\mathrm{Hess}(Q_1), \ldots, \mathrm{Hess}(Q_l)\}$ also generate the module $C^\infty(\Sigma,\mathrm{Sym}^2\Sigma^*)^{W_0}$ over $C^\infty(V)^{W_0}$ --- see Lemma 3.1 in \cite{Field77}.

Since $\sigma$ is $W(\Sigma)$-equivariant, it is also $W(\Sigma)_0$-equivariant, and so there are smooth $W_0$-invariants $a_i$ such that
$$ \sigma=\sum_i a_i\cdot\mathrm{Hess}(Q_i)$$

By the Chevalley Restriction Theorem, there are unique extensions of $a_i$ and $Q_i$ to
$$ \tilde{a}_i\in C^\infty(V)^{K_0} \qquad \tilde{Q}_i\in \R[V]^{K_0}$$

Define $\tilde{\sigma}_0$ by
$$ \tilde{\sigma}_0=\sum_i \tilde{a}_i\cdot\mathrm{Hess}(\tilde{Q}_i)\quad \in C^\infty(V,\mathrm{Sym}^2V^*)^{K_0}$$
and $\tilde{\sigma}$ by
$$ \tilde{\sigma}=\frac{1}{|K/K_0|}\sum_{h\in K/K_0} h\cdot\tilde{\sigma}_0 \quad \in C^\infty(V,\mathrm{Sym}^2V^*)^{K}$$

We claim that $i^*\tilde{\sigma}=\sigma$. The idea is to show that $i^*(g\cdot \tilde{\sigma}_0)=\sigma$ for all $g\in K$.

Indeed, given $g\in K$, $g^{-1}\Sigma$ is another section, and since the action of $K_0$ is polar with the same sections as $K$, it must act transitively on the sections, and thus there is  $h\in K_0$ such that $g^{-1}\Sigma=h^{-1}\Sigma$, that is, $gh^{-1}\in N(\Sigma)$.

Then
$$ g\cdot \tilde{\sigma}_0=gh^{-1}h\cdot \tilde{\sigma}_0=gh^{-1}\cdot \tilde{\sigma}_0$$
because $\tilde{\sigma}_0$ is $K_0$-equivariant. Applying $i^*$ to both sides gives
$$ i^*(g\cdot \tilde{\sigma}_0)=[gh^{-1}]\cdot i^*(\tilde{\sigma}_0)=[gh^{-1}]\cdot\sigma=\sigma$$
where $[gh^{-1}]$ denotes the class of $gh^{-1}$ in the quotient $W=N(\Sigma)/Z(\Sigma)$. This finishes the proof that $i^*\tilde{\sigma}=\sigma$.

Now we turn to the second statement. Let $X,Y\in T_pV$ with $X$ vertical and $Y$ horizontal, that is, tangent and normal to the orbit through $p$. Extend them to parallel (in the Euclidean metric) vector fields, also called $X$ and $Y$.

For each $\tilde{Q_i}$, let $f=d\tilde{Q_i}(X)$. Since $\Sigma$ is a vector subspace, $X$ is orthogonal to $\Sigma$ at every point of $\Sigma$. So at every regular $q\in \Sigma$, $X(q)$ is tangent to the orbit. Since $\tilde{Q}_i$ is constant on orbits, $f=0$ at every regular $q\in\Sigma$, and hence on all of $\Sigma$.

In particular Hess$(\tilde{Q}_i)_p (X,Y)= df(Y)=0$. Since $\tilde{\sigma}_0$ is a linear combination of such Hessians, and $\tilde{\sigma}$ the average of $\tilde{\sigma}_0$ over $K$ , we get $\tilde{\sigma} (X,Y)=0$ as well.

Finally, assume $\sigma$ is positive-definite at $0\in\Sigma$. We will show that $\tilde{\sigma}$ is positive definite at $0\in V$.

It suffices to do so for
$$\tilde{\sigma}_0=\sum_{i=1}^l \tilde{a}_i \cdot\mathrm{Hess}(\tilde{Q}_i)$$

Consider the decomposition of $V$ into irreducible $K$-representations
$$ V=\R^m\oplus V_1\oplus\cdots\oplus V_k$$
where $K$ fixes $\R^m$ and each $V_i$ is non-trivial. Theorem 4 in \cite{Dadok85} implies that $V_i$ are pairwise inequivalent and polar, with sections $\Sigma_i=\Sigma\cap V_i$. Moreover $$ \Sigma = \R^m\oplus \Sigma_1\oplus\cdots\oplus \Sigma_k$$
with $\Sigma_i$ pairwise inequivalent $W$-representations.

By Schur's lemma,
$$ \sigma(0)=A\oplus\lambda_1 \mathrm{Id}_{\Sigma_1}\oplus\cdots\oplus \lambda_k\mathrm{Id}_{\Sigma_k}$$
where $A$ is a positive-definite symmetric $m\times m$ matrix, and each $\lambda_i>0$.
We can rewrite this as
$$ \sigma(0)=\sum _{1\leq i ,j\leq m}\frac{a_{ij}}{2}\mathrm{Hess}(x_ix_j)\oplus \frac{\lambda_1}{2}\mathrm{Hess}(P_1)\oplus\cdots\oplus \frac{\lambda_k}{2}\mathrm{Hess}(P_k)$$
where $P_i\in\R[\Sigma]^W$ is given by $P_i(v)=|P_{\Sigma_i}(v)|^2$, that is, the norm square of the component of $v$ in $\Sigma_i$. The unique extension of $P_i$ to a $K$-invariant polynomial on $V$ is clearly $\tilde{P_i}(v)=|P_{V_i}(v) |^2$.

We may assume that 
$$ \{ Q_i\ |\ 1\leq i\leq l\ ,\ \deg Q_i=2\}=\{x_ix_j\ |\ 1\leq i\leq j \leq m\}\cup\{ P_1,\ldots P_k\}$$
because the Hessians of both sets form a basis for $(\mathrm{Sym}^2\Sigma)^W$. Therefore
$$ \tilde{\sigma}_0(0)=A\oplus \lambda_1\mathrm{Id}_{V_1}\cdots\oplus\lambda_k\mathrm{Id}_{V_k}$$
is positive-definite because $A$ is positive-definite and each $\lambda_i>0$.
\end{proof}

Now we are ready for the general case:

\begin{proof}[Proof of Theorem \ref{extendingsym2tensors} using Proposition \ref{linearversion}]

Let $i:\Sigma\to M$ be a section, that is, an isometric immersion meeting all $G$-orbits orthogonally which does not factor through a subcover of $\Sigma$.

Using $G$-invariant partitions of unity on $M$, it is sufficient to construct the desired symmetric $2$-tensor $\tilde{\sigma}$ locally, in a neighborhood of an arbitrary $G$-orbit $\mathcal{O}\subset M$.

Let $q\in \Sigma$ with image $p=i(q)\in\mathcal{O}\subset M$. Denote by $K=G_p$ the isotropy, and $V=(T_p\mathcal{O})^{\perp}$ the slice at $p$. $V$ is a polar $K$-representation with section $(di)_q:T_q\Sigma\to V$. The generalized Weyl group acting on $T_q\Sigma$ is the isotropy $W(\Sigma)_q$.

For $r>0$ small enough, the Slice Theorem can be used to describe the $r$-neighborhood $\mathcal{U}$ of $\mathcal{O}$. It is $G$-equivariantly diffeomorphic to $G\times_KB(V) = (G\times B(V))/K$, where $B(V)=\{v\in V \ |\ |v|<r\}$, and $K$ acts on $G\times B(V)$ by 
$$ (k, (g,v))\mapsto (gk^{-1}, kv)$$

The pull-back $\tau=\exp_q^*(\sigma)$ of the given $\sigma$ to $B(T_q\Sigma)=\{v\in T_q\Sigma \ |\ |v|<r\}$ is a smooth $W(\Sigma)_q$-invariant symmetric $2$-tensor. By Proposition \ref{linearversion}, $\tau$ is the pull-back (restriction) of a smooth $K$-invariant symmetric $2$-tensor $\tilde{\tau}$ on $B(V)$.

Choose two arbitrary smooth $K$-equivariant maps
$$ \tau_1 : V\to \mathrm{Sym}^2(T_p^*\mathcal{O}),\quad
\tau_2 : V\to V^*\otimes (T_p^*\mathcal{O})$$
and define $\tilde{\sigma}_0=\tilde{\tau} +\tau_1 +\tau_2 : V\to \mathrm{Sym}^2 (T_p^*\mathcal{U})$.

Since the action of $K$ on $G\times V$ is free, $\tilde{\sigma}_0$ extends to a unique smooth $G$-invariant symmetric $2$-tensor $\tilde{\sigma}$ on $\mathcal{U}$.

By construction the pull-back $i^*(\tilde{\sigma})$ of $\tilde{\sigma}$ to $i^{-1}(\mathcal{U})$ is a $W(\Sigma)$-invariant tensor which agrees with the given $\sigma$ on the connected component of $i^{-1}(\mathcal{U})$ containing $q$. Since $W(\Sigma)$ acts transitively on such connected components, they must agree on all of $i^{-1}(\mathcal{U})$. This finishes the proof of the first (non-metric) statement.

To deal with the case where $\sigma$ is a metric, we follow the same steps as above, with two modifications. First, when $\tilde{\tau}$ is defined, we note that by Proposition \ref{linearversion}, $\tilde{\tau}$ is positive definite at $0$, and preserves the relation of orthogonality between $K$-orbits and the section. By shrinking $r>0$ if necessary, we may assume that $\tilde{\tau}$  is positive definite on $B(V)$.  Second, when defining $\tilde{\sigma}_0$, $\tau_1$ must be positive definite on $B(V)$ and $\tau_2$ must be identically zero. Together these imply that $\tilde{\sigma}$ is positive definite on all of $\mathcal{U}$, and defines a metric with respect to which $G$-orbits and $\Sigma\cap\mathcal{U}$ are orthogonal. 

\end{proof}

\section{The Hessian Theorem for finite reflection groups}
\label{sectionhessthm}
\subsection{General Remarks}
The goal of this section is to give a proof of the Hessian Theorem (see Theorem \ref{hessthm}).

Recall some facts about finite reflection groups: First, the algebra of invariants, $\R[V]^W$, is a free polynomial algebra with as many generators as the dimension of $V$. This is known as Chevalley's theorem --- see \cite {Bourbaki} Chapter V. Such a set of homogeneous generators is called a set of \emph{basic invariants}. Second, $V$ is reducible as a $W$-representation if and only if $V=V_1\times V_2$ and $W=W_1\times W_2$ for two reflection groups $W_k\subset O(V_k)$ --- see section 2.2 in \cite{Humphreys}. Because of the latter, the following proposition reduces the proof of the Hessian theorem to the irreducible case.

\begin{proposition}
\label{product}
Let $W_k\subseteq O(V_i)$, $k=1,2$ be two finite reflection groups in the Euclidean vector spaces $V_k$, and let $W=W_1\times W_2\subset O(V)=O(V_1\times V_2)$. Then the conclusion of the Hessian theorem holds for $W\subset O(V)$ if and only if it holds for both  $W_k\subseteq O(V_k)$, $k=1,2$.
\end{proposition}

\begin{proof}
Let $i_k:V_k\to V_1\times V_2$ and $p_k:V_1\times V_2\to V_k$ be the natural inclusions and projections. As a $W$-representation, $\mathrm{Sym}^2(V^*)$ decomposes as
$$ \mathrm{Sym}^2(V^*) =\mathrm{Sym}^2(V_1^*)\oplus \mathrm{Sym}^2(V_2^*)\oplus (V_1^*\otimes V_2^*)$$
Denote by $i_{11}$ and $i_{22}$ the natural inclusions of the first two summands. All these maps are $W$-equivariant.

\begin{enumerate}
\item Assume the conclusion of the Hessian Theorem holds for $W\subset O(V)$. Thus there are $Q_j\in\R[V]^W$ whose Hessians form a basis for $\R[V,\mathrm{Sym}^2(V^*)]^W$. Then the restrictions $Q_j|_{V_k}=i_k^*Q_j$ generate $\R[V_k,\mathrm{Sym}^2(V_k^*)]^{W_k}$ as an $\R[V_k]^{W_k}$-module.

    Indeed, given $\sigma\in \R[V_k,\mathrm{Sym}^2(V_k^*)]^{W_k} $, define $$\tilde{\sigma}=i_{kk}\circ\sigma\circ p_k$$ Since $\tilde{\sigma}$ is $W$-equivariant, there are $a_j\in \R[V]^W$ such that $\tilde{\sigma}=\sum_j a_j\mathrm{Hess}(Q_j)$. Therefore
    $$ \sigma=i_k^*(\tilde{\sigma})=\sum_j (a_j|_{V_k})\mathrm{Hess}(Q_j|_{V_k})$$

\item Assume the conclusion of the Hessian Theorem holds for $W_k\subset O(V_k)$. Let $\rho_j\in\R[V_1]^{W_1}$, $j=1,\ldots n_1$ and $\psi_j\in\R[V_2]^{W_2}$, $j=1,\ldots n_2$ be basic invariants on $V_1$ and $V_2$ respectively, and $Q_j\in\R[V_1]^{W_1}$, for $j=1,\ldots (n_1^2+n_1)/2$, $R_j\in\R[V_2]^{W_2}$, for $j=1,\ldots (n_2^2+n_2)/2$ be homogeneous invariants whose Hessians form a basis for the corresponding spaces of equivariant symmetric $2$-tensors.

Claim: The Hessians of the following set of $W=W_1\times W_2$-invariant polynomials on $V=V_1\times V_2$ form a basis for the space of equivariant symmetric $2$-tensors on $V$:
$$ \{Q_j\} \cup \{R_j\}\cup \{\rho_i\psi_j , \quad i=1\ldots n_1,\  j=1\ldots n_2\}$$

Indeed, $\R[V,\mathrm{Sym}^2(V^*)]^W$ decomposes as
$$ \R[V,\mathrm{Sym}^2(V_1^*)]^W\oplus \R[V,\mathrm{Sym}^2(V_2^*)]^W\oplus\R[V,V_1^*\otimes V_2^*]^W$$
The first two pieces are freely generated over $\R[V]^W$ by Hess$Q_j$ and Hess$R_j$. The third piece can be rewritten as $\R[V,V_1^*\otimes V_2^*]^W=\R[V_1,V_1^*]^{W_1}\otimes\R[V_2,V_2^*]^{W_2}$. By Solomon's theorem, $\R[V_k,V_k^*]^{W_k}$ are freely generated by $d\rho_j$ and $d\psi_j$, so that $\R[V,V_1^*\otimes V_2^*]^W$ is freely generated by $d\rho_j\otimes d\psi_j$. To finish the proof of the Claim one uses the product rule
$$ \mathrm{Hess}(\rho_i\psi_j)= d\rho_i\otimes d\psi_j + \rho_i\mathrm{Hess}(\psi_j)+\psi_j\mathrm{Hess}(\rho_i) $$

\end{enumerate}
\end{proof}

The proposition above reduces the proof of the Hessian theorem to the irreducible case. The next step is to use the classification of finite irreducible reflection groups according to their Coxeter graphs, and prove the Hessian theorem for each of them.

\sloppy

The proofs for each type all fit in the following general framework:
\label{generalframework}
\begin{itemize}
\item Recall that the Poincar\'e series of a graded vector space $M=\bigoplus_{i=0}^\infty M_i$ is defined by $P_t(M)=\sum_{i=0}^\infty \dim(M_i)t^i$. In our situation we put $M=\R[V,\mathrm{Sym}^2V^*]^W$ and compute its Poincar\'e series using Molien's formula (see section 3.1 in \cite{NeuselSmith} ):
\begin{theorem2} [Molien's formula] %\label{Molien}
 Let $G$ be a finite group, $\rho:G \to \mathrm{GL}(V)$ a representation, and $U$ another representation with character $\chi$. Then
$$ P_t(\R[V,U]^G)=\frac{1}{|G|}\sum_{g\in G}\frac{\chi(g)}{\det(\mathrm{Id}-t\rho(g))}$$
\end{theorem2}

 \item Note that $\R[V,\mathrm{Sym}^2V^*]^W$ is a free module over $\R[V]^W$, of rank equal to $l=\dim \mathrm{Sym}^2 V^*=n(n+1)/2$. This is a consequence of a theorem by Chevalley which says that $\R[V]$ is a free module over $\R[V]^W$, more precisely that $\R[V]=\R[V]^W\otimes \mathcal{I}$, where $\mathcal{I}$ is the regular representation of $W$ --- see \cite{Chevalley55}. Indeed, 
 $$ \R[V,\mathrm{Sym}^2V^*]^W=(\R[V]\otimes\mathrm{Sym}^2V^*)^W=\R[V]^W\otimes (\mathcal{I}\otimes\mathrm{Sym}^2V^*)^W$$
 and since $\mathcal{I}$ is the regular representation, $(\mathcal{I}\otimes\mathrm{Sym}^2V^*)^W$ is isomorphic to Sym$^2V^*$. Therefore $P_t(M)/P_t(\R[V]^W) $ is a polynomial which encodes the degrees of an $\R[V]^W$-basis for $M$.  Note also that  $\R[V]^W$ is free, so that $P_t(\R[V]^W)=\prod_{i=1}^n\frac{1}{1-t^{d_i}}$ where $d_i$ are the \emph{degrees} of $W$, that is, $d_i=\deg (\rho_i)$.
 
 \item We choose an appropriate subset $$T\subset \{\rho_i\ |\ 1\leq i\leq n\}\cup\{\rho_i\rho_j\ |\ 1\leq i\leq j\leq n\}$$ with $l=n(n+1)/2$ elements such that the degrees of $\{\mathrm{Hess}(Q), Q\in T \}$ are the same as those of a basis.

\item At this point it is enough to show that the chosen Hessians are linearly independent over $\R[V]^W$, because then the submodule generated by them will have the same Poincar\'e series as $M$, forcing them to be equal for dimension reasons in each degree. In order to show that such a set of tensors is linearly independent we show that their values at a particular vector $v\in V$ are linearly independent over $\R$. Here $v$ can be any regular vector.
\end{itemize}

\fussy

Now we follow the program outlined above first for the dihedral groups, then the reflection groups of classical type, that is, types A, B and D, and finally for the groups of exceptional type, namely $H_3$, $H_4$, $F_4$, $E_6$, $E_7$ and $E_8$. For the exceptional groups we need to use a computer for some of the calculations.

\subsection{Dihedral Groups}
Let $V=\R^2$ and $W=D_n$, the dihedral group with $2n$ elements. It is generated by $\{a,b\}$, where $a$ is counterclockwise rotation by $2\pi /n$ and $b$ is the reflection across the $x$-axis. Thus
$$ D_n=\{1,a,\ldots, a^{n-1},b,ab,\ldots a^{n-1}b\}$$

It is convenient to introduce complex notation: identify $(x,y)$ with $z=x+iy$. Then $a$ becomes complex multiplication with $\xi=e^{2\pi i/n}$, and $b$ becomes complex conjugation.

It is well known (see \cite{NeuselSmith} page 108) that we can take the basic invariants to be
$$ \rho_1(x,y)=z\bar{z}=x^2+y^2 \qquad \rho_2(x,y)=\mathrm{Re}(z^n)$$

Thus the degrees of $W$ are $d_1=2$, $d_2=n$, and
$$ P_t(\R[V]^W)=\frac{1}{(1-t^2)(1-t^n)}$$

\begin{proposition}
\label{dihedralPoincare}
For $n\geq 2$ let $W$ be the dihedral group with $2n$ elements acting on $V=\R^2$. The Poincar\'e series for the space of equivariant symmetric $2$-tensors is
$$P_t (\R[V,\mathrm{Sym}^2V^*]^W)=(1+t^2+t^{n-2})P_t(\R[V]^W) $$
\end{proposition}

\begin{proof}
Since the eigenvalues of $a^j$ are $\xi^j$ and $\xi^{-j}$, we see that
$$ \det(1-ta^j)=(t-\xi^j)(t-\xi^{-j}) \qquad \mathrm{tr}(\mathrm{Sym}^2 a^j)=1+\xi^{2j}+\xi^{-2j}$$

Since $a^jb$ are reflections,
$$ \det(1-ta^jb)=1-t^2 \qquad \mathrm{tr}(\mathrm{Sym}^2 a^jb)=1$$

Applying Molien's formula gives
$$ \frac{P_t (\R[V,\mathrm{Sym}^2V^*]^W)}{P_t(\R[V]^W)}=\frac{(t^2-1)(t^n-1)}{2n}\left(\frac{n}{1-t^2}+\sum_{j=0}^{n-1}\frac{1+\xi^{2j}+\xi^{-2j}}{(t-\xi^j)(t-\xi^{-j})}\right)$$

Denoting this polynomial by $f(t)$, we want to show that $f(t)=1+t^2+t^{n-2}$. Since both have degrees less than or equal to $n$, it is enough to check that they have the same values at $n+1$ distinct points, which we can take to be $t=0,1,\xi,\xi^2,\ldots \xi^{n-1}$.

Start with $f(0)=\frac{1}{2n}(2n + 2\sum_j\xi^{2j})$. If $n=2$ this equals $2$, otherwise $\xi^2\neq 1$ and $f(0)=1$. In both cases the value of $f$ agrees with that of $1+t^2+t^{n-2}$. Next consider $t=\xi^k\neq 1,-1$. Only the terms $j=k,-k$ in the sum survive, and they are equal:

\begin{align*}
f(\xi^k) &= \frac{1}{2n}2(1+\xi^{2k}+\xi^{-2k})\cdot\left.\frac{(t^2-1)(t^n-1)}{(t-\xi^k)(t-\xi^{-k})}\right|_{t=\xi^k}\\
 &= \frac{(1+\xi^{2k}+\xi^{-2k})(\xi^{2k}-1)}{n(\xi^k-\xi^{-k})} \cdot\left.\frac{d(t^n-1)}{dt}\right|_{t=\xi^k}\\
  &=1+\xi^{2k}+\xi^{-2k} = (1+t^2+t^{n-2})|_{t=\xi^k}
\end{align*} 
as wanted. The remaining cases $t=1,-1$ are similar.
\end{proof}

\begin{theorem3}
\label{dihedral}
In the notations above, a basis for $\R[V,\mathrm{Sym}^2V^*]^W$ as a free $\R[V]^W$-module is
$$ \{\mathrm{Hess}(\rho_1), \mathrm{Hess}(\rho_1^2), \mathrm{Hess}(\rho_2)\}$$
\end{theorem3}

\begin{proof}
It's more convenient and enough to do the computation of the Hessians in the basis $z,\bar{z}$ instead of $x,y$:
$$ \mathrm{Hess}(z\bar{z})=\begin{bmatrix}
                            0 & 1 \\ 1 & 0
                           \end{bmatrix},\quad
\mathrm{Hess}(z^2\bar{z}^2)=\begin{bmatrix}
                            2\bar{z}^2 & 4z\bar{z} \\ 4z\bar{z} & 2z^2
                           \end{bmatrix} $$ $$
\mathrm{Hess}\left(\frac{z^n+\bar{z}^n}{2}\right)=\frac{n(n-1)}{2}\begin{bmatrix}
                            z^{n-2} & 0 \\ 0 & \bar{z}^{n-2}
                           \end{bmatrix}
$$

Listing the upper triangular entries of the matrices above into a $3\times 3$ matrix, we get:

$$
\begin{bmatrix}
 0 & 1 & 0\\ 2\bar{z}^2 & 4z\bar{z} & 2z^2 \\ \frac{n(n-1)}{2} z^{n-2} & 0 & \frac{n(n-1)}{2} \bar{z}^{n-2}
\end{bmatrix}
 $$

Since the determinant of the matrix above is $-2n(n-1)i\ \mathrm{Im}(z^n)\ne 0$, the Hessians are linearly independent over $\R[V]$, hence also over $\R[V]^W$, and since their degrees are $0$, $2$ and $n-2$, they must form a basis of $\R[V,\mathrm{Sym}^2V^*]^W$ by Proposition \ref{dihedralPoincare}.

\end{proof}
\subsection{Classical Groups}

We will use a combinatorial identity called the cycle index formula. Generating functions-type manipulations with this identity together with Molien's formula yield the Poincar\'e series of the symmetric  $2$-tensors.

Making a particular choice of basic invariants and of a regular vector $v$, the Hessians of the $n$ basic invariants evaluated at $v$ have particularly simple forms. Looking at these $n$ matrices it is then easy to see which sets of Hessians are linearly independent when evaluated at $v$.

We start by recalling the definitions and fixing notations. Let $V=V_n=\R^n$, with standard basis $e_1,\ldots, e_n$, which we declare to be orthonormal. Define $W=W_n\subset O(V)$ by:
\begin{itemize}
\item $W=S_n$, the group of permutation matrices, for type A;
\item $W=S_n\ltimes \{\pm 1\}^n$, the group of signed permutation matrices, for type B; and
\item $W=\{(\sigma, (\epsilon_1,\ldots, \epsilon_n))\in S_n\ltimes \{\pm 1\}^n \ |\ \epsilon_1\cdots\epsilon_n=1\}$ for type D.
\end{itemize}

Note: Strictly speaking the standard representation of the Weyl group of type $A_{n-1}$ is the subrepresentation of $S_n$ on $(1,\ldots,1)^\perp\subset V$, but by Proposition \ref{product} it is enough to prove the Hessian Theorem for the reducible representation $V$ instead.

Recall that a permutation $\sigma\in S_n$ can be decomposed into disjoint cycles, and denote by $k_i(\sigma)$ the number of cycles of length $i$. For a signed permutation $g=(\sigma,\epsilon)$, where $\epsilon:\{1,2,\ldots n\}\to\{\pm 1\}$, the cycles of $\sigma$ are divided into positive and negative, according to the sign of the product of the corresponding values of $\epsilon$. Denote by $k_i^+(g)$ (resp. $k_i^-(g)$) the number of positive (resp. negative) cycles of length $i$, so that $k_i(\sigma)=k_i^+(g)+k_i^-(g)$.

Our computations of the Poincar\'e series will rely on the Cycle Index Formula, a combinatorial identity of generating functions in the formal variables $z$, and $y_1,\ldots y_n$ (resp. $y_1^+,\ldots,y_n^+,y_1^-, \ldots y_n^-$) which counts the number of permutations (resp. signed permutations) having a given cycle decomposition: (see \cite{Wilf} chapter 4.7)

\begin{theorem2}[Cycle index formula for $S_n$]
%\label{cycleindex}
$$  \sum_{n=1}^\infty \frac{z^n}{n!}\sum_{g\in S_n}\mathbf{x}^{\mathbf{k}}=\mathrm{exp}\left( \sum_{j=1}^\infty \frac{z^jx_j}{j} \right)$$
where $\mathbf{x}^{\mathbf{k}}$ denotes $x_1^{k_1(g)}x_2^{k_2(g)}\cdots x_n^{k_n(g)}$.
\end{theorem2}

From the formula above in type A, it is easy to prove the corresponding formulas in types B and D:
\begin{proposition}
 The cycle index formulas for types B and D are:
 \begin{itemize}

 \item Type B:
 $$  \sum_{n=1}^\infty \frac{z^n}{n!2^n}\sum_{g\in S_n\ltimes (\pm 1)^n}\mathbf{x}^{\mathbf{k}}=\mathrm{exp}\left( \sum_{j=1}^\infty \frac{z^j(x_j^+ +x_j^-)}{2j} \right)$$
where $\mathbf{x}^{\mathbf{k}}$ denotes $\prod_{j=1}^n(x_j^+)^{k_j^+(g)}(x_j^-)^{k_j^-(g)}$.

 \item Type D:
$$\sum_{n=1}^\infty\frac{z^n}{n!2^{n-1}} \sum_{g\in S_n\ltimes H_n}\mathbf{x}^{\mathbf{k}}=
\exp\left( \sum_{j=1}^\infty \frac{z^j(x_j^++x_j^-)}{2j} \right)+
\exp\left( \sum_{j=1}^\infty \frac{z^j(x_j^+-x_j^-)}{2j} \right) $$
where $\mathbf{x}^{\mathbf{k}}$ denotes $\prod_{j=1}^n(x_j^+)^{k_j^+(g)}(x_j^-)^{k_j^-(g)}$, and $H_n=\{\epsilon\in\{\pm 1\}^n \ \ |\ \ \epsilon_1\cdots\epsilon_n=+1\}$
 \end{itemize}

\end{proposition}

\begin{proof}
We start with type B. Fix $\tau\in S_n$ and consider the inner sum in
$$\sum_{g\in S_n\ltimes (\pm 1)^n}\mathbf{x}^{\mathbf{k}}= \sum_{\tau\in S_n} \ \sum_{\epsilon\in (\pm 1)^n} \mathbf{x}^{\mathbf{k}} $$
The sign of each cycle $(i_1\ldots i_l)$ in $\tau$ depends only on the values of $\epsilon_1\ldots\epsilon_l$, and is actually positive for $2^{l-1}$ such values and negative for the remaining $2^{l-1}$. Thus this inner sum becomes
$$ \prod_{\text{cycles }\sigma}(2^{l(\sigma)-1}x^-_{l(\sigma)} +2^{l(\sigma)-1}x^+_{l(\sigma)} )=
2^n\prod_{j=1,\ldots n}\left(\frac{x_k^+ +x_j^-}{2}\right)^{k_j}$$
Now the result follows from the cycle index formula for $S_n$.

Now for type D, let $A$ denote the quantity $$A=\frac{1}{n!2^{n-1}} \sum_{g\in S_n\ltimes H_n}\mathbf{x}^{\mathbf{k}}$$ which we want to compute, and $B$ the complementary sum $$B=\frac{1}{n!2^{n-1}} \sum_{g\in (S_n\ltimes \{\pm 1\}^n-S_n\ltimes H_n)}\mathbf{x}^{\mathbf{k}}$$

Using the cycle index formula for type B one sees that
$$ A+B=2\exp\left( \sum_{j=1}^\infty \frac{z^j(x_j^++x_j^-)}{2j} \right)$$
$$ A-B=2\exp\left( \sum_{j=1}^\infty \frac{z^j(x_j^+-x_j^-)}{2j} \right)$$
and thus $A=(A+B+A-B)/2$ is as stated.

\end{proof}

Note that the terms appearing in Molien's formulas,
$$ P_t(\R[V]^W)=\frac{1}{|W|}\sum_{g\in W}\frac{1}{\det(1-t g)}$$
$$ P_t(\R[V,\mathrm{Sym}^2V^*]^W)=\frac{1}{|W|}\sum_{g\in W}\frac{\chi(g)}{\det(1-t g)}$$

depend only on $k_i(g)$ (resp. $k_i^+(g)$ and $k_i^-(g)$) . Indeed,
\begin{itemize}
\item In Type A, for a permutation matrix $g$, $ \mathrm{tr}(g)=k_1 $ and

$$\chi(g)=\mathrm{tr}(\mathrm{Sym}^2g)=\frac{\mathrm{tr}(g^2)+(\mathrm{tr}(g))^2}{2}= \frac{k_1^2}{2} +\frac{k_1}{2} +k_2$$
$$ \det(1-tg)=(1-t)^{k_1}(1-t^2)^{k_2}\cdots (1-t^n)^{k_n}$$

\item In Types B and D, for a signed permutation matrix $g$, $ \mathrm{tr}(g)=k_1^+-k_1^-$ and
$$\chi(g)=\mathrm{tr}(\mathrm{Sym}^2 (g))=k_1^+ +k_1^- +\binom{k_1^+}{2} +\binom{k_1^-}{2}+k_2^+ -k_1^+k_1^--k_2^-$$
$$ \mathrm{det} (1-tg)=\prod_{i=1}^n(1-t^i)^{k_i^+} \prod_{i=1}^n(1+t^i)^{k_i^-}$$
\end{itemize}

Combining the formulas above for $ \mathrm{det} (1-tg)$ with Molien's Formula, one gets:
\begin{proposition}
Replacing $x_j$ with $(1-t^j)^{-1}$ (resp. $x_j^+$ with $(1-t^j)^{-1}$ and $x_j^-$ with $(1+t^j)^{-1}$) in the right hand side of the Cycle Index Formula gives
$$\sum_{n=1}^\infty z^n P_t(\R[V_n]^{W_n})$$

\end{proposition}

To compute the Poincar\'e series of $\R[V,\mathrm{Sym}^2 V^*]^W$ in terms of the Poincar\'e series of $\R[V]^W$, we first apply an appropriate differential operator to the Cycle Index Formula, and then use Molien's formula together with the Proposition above. 

\begin{proposition}
\label{Poincare}
One has the following formulas for
$$ \frac{P_t(\R[V,\mathrm{Sym}^2V^*]^W)}{P_t(\R[V]^W)} $$
\begin{itemize}
\item In type A
$$ \frac{1-t^n}{1-t}+\frac{(1-t^{n-1})(1-t^n)}{(1-t)(1-t^2)}$$
\item In Type B
$$ \frac{1-t^{2n}}{1-t^2}+\frac{(1-t^{2n-2})(1-t^{2n})t^2}{(1-t^2)(1-t^4)} $$
\item In type D
$$ \frac{1-t^{2n}}{1-t^2} +\frac{(t^2+t^{n-2})(1-t^{2n-2})(1-t^n)}{(1-t^2)(1-t^4)} $$
\end{itemize}
\end{proposition}
\begin{proof}
We define a differential operator $D$ such that
$$D\mathbf{x}^{\mathbf{k}}=\mathrm{tr}(\mathrm{Sym}^2 (g))\mathbf{x}^{\mathbf{k}} $$
Namely,
\begin{itemize}
\item In type A,
$$ D=\frac{x_1^2}{2}\frac{\partial^2}{\partial x_1^2}+x_1\frac{\partial}{\partial x_1}+x_2\frac{\partial}{\partial x_2}$$
\item In types B and D,
$$ D= x_1^+\frac{\partial}{\partial x_1^+} +
x_1^-\frac{\partial}{\partial x_1^-} +
\frac{(x_1^+)^2}{2}\frac{\partial^2}{\partial (x_1^+)^2}+
\frac{(x_1^-)^2}{2}\frac{\partial^2}{\partial (x_1^-)^2}+
x_2^+\frac{\partial}{\partial x_2^+}-
x_1^+x_1^-\frac{\partial^2}{\partial x_1^+\partial x_1^-}-
x_2^-\frac{\partial}{\partial x_2^-}$$
\end{itemize}

Apply $D$ to both sides of the appropriate cycle index formula. Then replace $x_j$ with $(1-t^j)^{-1}$ (resp. $x_j^+$ with $(1-t^j)^{-1}$ and $x_j^-$ with $(1+t^j)^{-1}$). By Molien's formula the left hand side is exactly
$$ \sum_{n=1}^\infty z^n P_t(\R[V_n,\mathrm{Sym}^2V_n^*]^{W_n})$$

Now we turn to the right-hand side. First we use the preceding Proposition. For example in type A we get
$$\left(\sum_{n=1}^\infty z^n P_t(\R[V_n]^{W_n})\right)\left( \frac{z^2}{2(1-t^2)} +\frac{z}{(1-t)} +\frac{z^2}{(1-t)^2} \right)$$
and similarly for types B and D. Then we use the fact that
$$P_t(\R[V]^W)=\prod_{i=1}^n\frac{1}{1-t^{d_i}}$$
where $d_1, \ldots d_n$ are the degrees of $W$, which in type A are $1,2,\ldots n$, in type B $2,4,\ldots 2n$ and in type D $2,4,\ldots 2n-2, n$. Finally we take coefficients of $z^n$ on both sides and simplify to get the stated formulas.
\end{proof}

Having calculated the Poincar\'e series of $\R[V,\mathrm{Sym}^2V^*]^W$, we know the degrees (with multiplicities) of the elements in a basis for this module over $\R[V]^W$. The next step is to define explicit invariants whose Hessians have the degrees we just found, and prove that they are linearly independent. We start by fixing sets $\{ \rho_1,\ldots\rho_n\}$ of basic invariants:
\begin{itemize}
\item For type A, $ \rho_j=\frac{1}{j} \sum_{i=1}^n (x_i)^j$;
\item For type B, $ \rho_j=\frac{1}{2j}\sum_{i=1}^n(x_i)^{2j} $; and
\item For type D, $ \rho_j=\frac{1}{2j}\sum_{i=1}^n (x_i)^{2j} $ if $j<n$ and $ \rho_n=x_1\cdots x_n$.
\end{itemize}

In order to identify which sets of Hessians of invariants are linearly independent we evaluate them at a regular vector $v$. We find that when $v$ is a regular eigenvector of a Coxeter element (see \cite{Humphreys}, section 3.16), the Hessians take a simple form. In practice we take the following $v$:
\begin{itemize}
\item For type A, $v=(1,\xi,\xi^2,\ldots ,\xi^{n-1})$, where $\xi=\mathrm{exp}(2\pi i/n)$;
\item For type B, $v=(1,\xi,\xi^2,\ldots ,\xi^{n-1})$, where $\xi=\mathrm{exp}(2\pi i/(2n))$; and
\item For type D, $v=(1,\xi,\xi^2,\ldots, \xi^{n-2},0)$, where $\xi=\exp(2\pi i/(2n-2))$.
\end{itemize}

Note: Strictly speaking these vectors are in the complexification $\mathbb{C}\otimes V$, not in $V=\R^n$. But this is irrelevant as far as linear independence is concerned.

Define the map $\rho :V\to \R^n$ by $\rho(x) =(\rho_1(x),\ldots,\rho_n(x))$. Since $\rho_i$ are a set of basic invariants, any invariant polynomial is of the form $\rho^*f$ for some $f\in\R[y_1,\ldots y_n]$. Then we have:

\begin{proposition}[Chain Rule]
Let $f\in\R[y_1,\ldots y_n]$ be a polynomial map on $\R^n$. In the coordinates $x_i$ of $V$ dual to the standard basis $e_i$, we have:
$$ \mathrm{Hess}(\rho^*f)=\left(\frac{\partial^2 \rho^*f }{\partial x_i \partial x_j} \right)_{i,j}=J^t(\rho^*(\mathrm{Hess}f)) J +\sum_{k=1}^n\left(\rho^*\left(\frac{\partial f}{\partial Y_k}\right)\right)\mathrm{Hess}(\rho_k)$$
where $J$ is the Jacobian matrix $\left( \frac{\partial\rho_i}{\partial x_j} \right)_{i,j}$.
\end{proposition}

\begin{proposition}
Let $M_k$ denote the $k\times k$ matrix with ones in the antidiagonal and zeros everywhere else. Here are the matrices $(J^T)^{-1}\mathrm{Hess}(\rho_i)J^{-1}$ evaluated at $v$ chosen above:
\begin{itemize}
\item In type A,
$$\frac{i-1}{n}\begin{bmatrix}
 M_{i-1} & 0 \\
0 & M_{n-i+1}
\end{bmatrix}$$
\item In type B,
$$\frac{2i-1}{n}
\begin{bmatrix}
 M_{i-1} & 0 \\
0 & M_{n-i+1}
\end{bmatrix}$$
\item In type D,
$$(2i-1)
\begin{bmatrix}
\frac{1}{n-1}M_{i-1} & 0 & 0\\
0 & \frac{1}{n-1}M_{n-i} &0\\
0 & 0 & (-1)^n
\end{bmatrix}$$
for $1\leq i\leq n-1$ and
$$ (J^t)^{-1}\mathrm{Hess}(\rho_n)J^{-1}=
\begin{bmatrix}
 0 & 0\\
0 & M_2
\end{bmatrix}$$
\end{itemize}
\end{proposition}

\begin{proof}
For type A, since $ \rho_i=\frac{1}{i} \sum_{j=1}^n (x_j)^i$ and $v=(1,\xi,\xi^2,\ldots ,\xi^{n-1})$, $J(v)$ is a Vandermonde matrix whose $(a,b)$ entry equals $J_{ab}=\xi^{(a-1)(b-1)}$; and Hess$(\rho_i)(v)=(2i-1)\cdot \mathrm{diag}(1,\xi^{i-2},\ldots \xi^{(i-2)(n-1)})$.

We want to show that 
$$n \cdot\mathrm{diag}(1,\xi^{i-2},\ldots \xi^{(i-2)(n-1)})=J^T \begin{bmatrix}
 M_{i-1} & 0 \\
0 & M_{n-i+1}
\end{bmatrix}J$$
Indeed, the $(l,m)$ entry on right-hand side is equal to 
$$ \sum_{a,b=1}^n J_{l,a} N_{a,b} J_{b,m}$$
where $N_{a,b}=1$ if $a+b=i$ or $n+i$, and $0$ otherwise. Thus the sum becomes
\begin{align*} &\sum_{a+b=i, i+n} \xi^{(l-1)(a-1) + (b-1)(m-1)}\\
&=\sum_{a=1}^{i-1}\xi^{(l-1)(a-1)+(i-a-1)(m-1)} + \sum_{a=i}^{n}\xi^{(l-1)(a-1)+(n+i-a-1)(m-1)}\\ 
&= \sum_{a=1}^{n}\xi^{(l-1)(a-1)+(i-a-1)(m-1)}\\
&=\xi^{(i-1)(m-1)-(l-1)}\sum_{a=1}^n\xi^{(l-m)a}
\end{align*}
In the sum over $a=i,\ldots n$ we have replaced $(n+i-a-1)$ with $(i-a-1)$ because $\xi^n=1$.
The last sum $\sum_{a=1}^n\xi^{(l-m)a}$ equals $n$ if $\xi^{l-m}=1$, that is, $l=m$; and zero otherwise. In the first case $(i-1)(m-1)-(l-1)=(i-2)(m-1)$, as wanted.

The proofs for types B and D are similar.
\end{proof}

\begin{theorem3}
\label{classical}
With the notations above, let 
$$ T=\{\rho_i\ |\ 1\leq i\leq n\}\cup\{\rho_i\rho_j\ |\ 1\leq i\leq j\leq n\}-T_0$$
where $T_0$ is a subset with $n$ elements satisfying:
\begin{itemize}
\item in type A, $T_0$ contains $\rho_1$ and one $\rho_i\rho_j$ with $i+j=k$ for each $k=n+2,\ldots 2n$;
\item in type B, $T_0$ has one $\rho_i\rho_j$ with $i+j=k$ for each $k=n+1,\ldots 2n$;
\item In type D, $T_0$ has $\rho_{n-1}\rho_n$ and one $\rho_i\rho_j$ with $i+j=k$ and $i,j\leq n-1$ for each $k=n,\ldots 2n-2$;

\end{itemize}

Then
$$ \{ \mathrm{Hess}(Q), \quad Q\in T\}$$
is a basis for $\R[V,\mathrm{Sym}^2V^*]^W$ as a free module over $\R[V]^W$.
\end{theorem3}

\begin{proof}
We need to show two things: First, that the chosen Hessians have the correct degrees, as dictated by Proposition \ref{Poincare}; and second, that they are linearly independent over $\R[V]^W$. Together these imply that $\R[V,\mathrm{Sym}^2V^*]^W$ and the submodule generated by these Hessians have the same Poincar\'e series and thus must coincide.

\begin{enumerate}
\item Showing that  $ \{ \mathrm{Hess}(Q), \quad Q\in T\}$ have the correct degrees consists of an algebraic manipulation of the formulas given in Proposition \ref{Poincare}. For example in type B,
$$\qquad\quad\frac{1-t^{2n}}{1-t^2}+\frac{(1-t^{2n-2})(1-t^{2n})t^2}{(1-t^2)(1-t^4)}=
t^{-2}\sum_{i=1}^{n} t^{2i} + t^{-2}\!\!\!\!\sum_{1\leq i\leq j\leq n}\!\!\!\!t^{2(i+j)} -t^{-2}\sum_{k=n+1}^{2n}\!\! t^{2k}$$
Indeed, on the right hand side of this equation the first two sums correspond to $\{\rho_i\}$ and $\{\rho_i\rho_j\}$ and the last sum to $T_0$. We omit the proofs for the other two types, which are similar.

\item To show they are linearly independent, we will evaluate at the regular vector $v\in V$ and use the preceding Proposition.

    The chain rule says that
    $$ \mathrm{Hess}(\rho_i\rho_j)=J^tE_{ij} J +\rho_i\mathrm{Hess}(\rho_j)+\rho_j\mathrm{Hess}(\rho_i)$$
    where $E_{ij}=$Hess$(y_iy_j)$.

    Since all the basic invariants belong to $T$ (except for type A, where $\rho_1$ is linear and thus has zero Hessian), we may replace each Hess$(\rho_i\rho_j)$ with $J^tE_{ij} J$. Since $J(v)$ is invertible, we can multiply with  $(J^T)^{-1}$ on the left and with $J^{-1}$ on the right and are left to prove that the following set of matrices is linearly independent:
    $$\mathcal{S}=\{E_{ij} \ |\ \rho_i\rho_j\in T \}\cup\{(J^t)^{-1}\mathrm{Hess}(\rho_i)J^{-1}\ |\ \rho_i\in T\} $$
This follows from the shapes of the matrices $(J^t)^{-1}\mathrm{Hess}(\rho_i)J^{-1}$ we computed in the previous Proposition, together with the fact that $\{ E_{ij}\ |\ 1\leq i\leq j\leq n\}$ is a basis (over $\R$) for the space os symmetric matrices.

For example in type B, from the description of $T$ given in the statement, we can think of $\mathcal{S}$ as obtained from $\{ E_{ij} \}$ thus: For each $k=n+1, \ldots, 2n$, replace one $E_{ij}$ where $i+j=k$ with 
$$(J^T)^{-1}\mathrm{Hess}(\rho_{k-n})J^{-1}= \frac{2(k-n)-1}{n}\cdot \frac{1}{2}\sum_{a+b=k-n,\  k}\!\!\!\! E_{ab}$$
Note that the removed $E_{ij}$ appears in this sum. Ordering $\mathcal{S}$ and $\{ E_{ij} \}$ appropriately, the matrix of $\mathcal{S}$ in terms of $E_{ij}$ is upper triangular with non-zero diagonal, so that $\mathcal{S}$ is linearly independent.

Similarly for types A and D. 
\end{enumerate}
\end{proof}

Note that the method used to compute the Poincar\'e series of $\R[V,\mathrm{Sym}^2 V^*]^W$ in Proposition \ref{Poincare} applies equally well to any type of tensor. In type A one simply  writes the character as a polynomial in the numbers $k_i$, and defines a differential operator $D$ by replacing each $k_i$ in this polynomial with $x_i\frac{\partial}{\partial x_i}$, and multiplication with composition of operators. Similarly for types B and D.

\subsection{Exceptional Groups}

Finally we prove the Hessian Theorem for the six exceptional finite reflection groups $W\subset O(V)$ usually called by the names of their Dynkin diagrams: $H_3$, $H_4$, $F_4$, $E_6$, $E_7$ and $E_8$. Note that the subscript denotes the rank $n=$dim$(V)$. 
In all cases our proof relies on calculations performed by a computer running GAP 3 (see \cite{GAP3}) using the package CHEVIE, which ultimately rely only on integer arithmetic. For the actual code that was used, see 

http://www.nd.edu/\~{}rmendes/sym2.txt

Recall a way of describing $W\subset O(V)$ from its Cartan matrix $C=(C_{ij})$. $V$ has a basis $r_1,\ldots r_n$ of simple roots with corresponding co-roots $r^{\vee}_1,\ldots r^{\vee}_n$. This means that $W$ is generated by the reflections in the hyperplanes $\ker(r^{\vee}_i)$ given by:
$$ R_i :v \mapsto v-r^{\vee}_i (v) r_i \qquad i=1,\ldots n$$ 
Expressing $v\in V$ in the basis of simple roots $v=a_1r_1+\ldots a_nr_n$, we get
$$R_i(v)=v-\left(\sum_j a_jr_i^{\vee}(r_j)\right)r_i$$
The coefficients $r_i^{\vee}(r_j)=C_{ij}$ form the Cartan matrix.

Here are the Cartan matrices for $H_3$, $H_4$ and $F_4$: (where $\zeta = \exp(2\pi i/5)$)

$$H_3:\  \left( \begin{array}{ccc} 
2 & \zeta^2+\zeta^3 & 0 \\
\zeta^2+\zeta^3 & 2 & -1 \\
0 & -1 & 2 
\end{array}\right)  ,\quad
H_4:\ \left( \begin{array}{cccc} 
2 & \zeta^2+\zeta^3 & 0 &0 \\
\zeta^2+\zeta^3 & 2 & -1 & 0 \\
0 & -1 & 2 & -1 \\
0 & 0 & -1 & 2
\end{array}\right) $$ $$
F_4:\ \left( \begin{array}{cccc} 
\phantom{-}2 & -1 & \phantom{-}0 & \phantom{-}0 \\
-1 & \phantom{-}2 & -1 & \phantom{-}0 \\
\phantom{-}0 & -2 & \phantom{-}2 & -1 \\
\phantom{-}0 & \phantom{-}0 & -1 & \phantom{-}2
\end{array}\right)$$
For the Cartan matrices in type E, refer to the tables at the end of \cite{Bourbaki}.

We start the proof of the Hessian theorem by describing how the program computes the polynomial
$$\frac{P_t(\R[V,\mathrm{Sym}^2V^*]^W)}{P_t(\R[V]^W)}$$
(see the general framework on page \pageref{generalframework})

We need to recall a few facts. Let $I$ be the ideal in $\R[V]$ generated by the homogeneous invariants of positive degree. The quotient $\R[V]/I$ is known to be isomorphic, as a $W$-representation, to the regular representation (see Theorem B in \cite{Chevalley55}), but it is also a graded vector space. Fixing an irreducible representation/character $\xi$, the Poincar\'e polynomial FD$_\xi(t) $ of the subspace of $\R[V]/I$ with components isomorphic to $\xi$ is called the \emph{fake degree} of $\xi$ . Moreover $\R[V]$ is isomorphic to $(\R[V]/I)\otimes\R[V]^W$ . Thus the Poincar\'e series of the vector subspace in $\R[V]$ given by the direct sum of all irreducible subspaces isomorphic to $\xi$ equals FD$_\xi(t) P_t(\R[V]^W)$.

The way the program computes $P_t(\R[V,\mathrm{Sym}^2V^*]^W)$ is as follows:

It first computes the character $\chi$ of Sym$^2V^*$, and decomposes it into a sum of irreducible characters, using character tables that come with CHEVIE.
$$ \chi=\sum_{\xi\text{ irreducible}} c_\xi \xi$$

It then uses a command in CHEVIE that returns the fake degrees of the irreducible characters $\xi$, and computes
$$ \sum_\xi c_\xi \mathrm{FD}_\xi(t)$$
Using Schur's lemma one sees that this equals
$$\frac{P_t(\R[V,\mathrm{Sym}^2V^*]^W)}{P_t(\R[V]^W)}$$

Here are the outputs:
$$ \begin{array}{l|l}
 & P_t(\R[V,\mathrm{Sym}^2V^*]^W) / P_t(\R[V]^W)=\\
 \hline\\
H_3 & t^{10}+t^8+t^6+t^4+t^2+1\\
H_4 & t^{38}+t^{30}+t^{28}+t^{22}+t^{20}+t^{18}+t^{12}+t^{10}+t^2+1\\
F_4 & t^{14}+t^{12}+2t^{10}+t^8+2t^6+t^4+t^2+1\\
E_6 &  t^{16} + t^{15} + t^{14} + t^{13} + 2t^{12} + t^{11} +2t^{10} + 2t^9 + \\
&+2t^8 + t^7 + 2t^6 + t^5 + t^4 + t^3 + t^2 + 1\\
E_7 &  t^{26} + t^{24} + 2t^{22} + 2t^{20} + 3t^{18} + 3t^{16} +
3t^{14}  + 3t^{12} +\\
&+ 3t^{10} + 2t^8 + 2t^6 + t^4 + t^2 + 1\\
E_8 & t^{46} + t^{42} + t^{40} + t^{38} + 2t^{36} + 2t^{34} + t^{32} + 3t^{30} + 2t^{28} + 2t^{26} + 3t^{24} + \\
& +2t^{22} +  2t^{20} + 3t^{18} + t^{16} + 2t^{14} + 2t^{12} + t^{10} + t^8 + t^6 + t^2 + 1 
\end{array}$$

Now we turn to the task of defining an explicit set of basic invariants $\rho_1,\ldots \rho_n\in \R[V]^W$. The degrees $d_i=\mathrm{deg}(\rho_i)$ are well known: (see tables at the end of \cite{Bourbaki})
$$ \begin{array}{l|l}
&\text{degrees } d_1,\ldots d_n\\
\hline
H_3 & 2,6,10\\
H_4 & 2,12, 20, 30\\
F_4 & 2,6,8,12 \\
E_6 & 2,5,6,8,9,12\\
E_7 & 2,6,8,10,12,14,18\\
E_8 & 2,8,12,14,18,20,24,30
\end{array}$$

We choose for each group a regular vector $v\in V$ and identify it with the row vector of its coefficients in the basis of the simple roots $r_i$. We also take one non-zero $\lambda\in V^*$ with minimal $W$-orbit size, namely the one which in the basis $\{r^{\vee}_i\}$ of simple co-roots is identified with the row vector
$$ \lambda=(0,\ldots 0,1)\cdot C^{-1} $$

Then the program computes the $W$-orbit $\mathcal{O}$ of $\lambda$. Here are our choices of $v$ and the number of elements in the orbit $\mathcal{O}$:
$$ \begin{array}{l|l|l}
 & v \text{ (in the basis } \{r_i\}) &  |\mathcal{O}|\\
\hline
H_3 & (1,2,3) & 12\\
H_4 & (1,2,3,5) & 20\\
F_4 & (2,-3,5,7)  & 24\\
E_6 & (2,-5,41,7,-9,110) & 27\\
E_7 & (2,-5,41,7,-9,110 ,-87) & 56\\
E_8 & (2,-5,41,7,-9,110 ,-87,11) & 240
\end{array}$$

Since $W$ permutes the linear polynomials in $\mathcal{O}$, for each natural number $m$ we get a $W$-invariant polynomial of degree $m$
$$\psi_m=\sum_{\lambda \in \mathcal{O}} \lambda^m$$

The invariants constructed this way are called the Chern classes associated to the orbit $\mathcal{O}$. See \cite{NeuselSmith} chapter 4.

\begin{proposition}
The polynomials $\rho_i=\psi_{d_i}$, $i=1,\ldots n$, form a set of basic invariants, and $v$ is indeed a regular vector.
\end{proposition}
\begin{proof}
Let $J$ be the Jacobian matrix
$$J= \left( \frac{\partial\rho_i}{\partial r^{\vee}_j} \right)_{i,j}=\left( \sum_{\lambda \in \mathcal{O}} d_i\lambda^{d_i-1}\frac{\partial\lambda}{\partial r^{\vee}_j}\right)_{i,j} $$  The program computes its determinant, evaluates it at the vector $v$, and checks that the value is non-zero. This proves both that $\rho_i$ are algebraically independent (see Proposition 3.10 in \cite{Humphreys}) and hence a set of basic invariants because they have the right degrees; and that $v$ is indeed a regular vector, that is, does not belong to any of the reflecting hyperplanes (see section 3.13 in \cite{Humphreys}). \end{proof}

We point out that L. Flatto and M. Weiner studied the set of all $\lambda\in V^*$ that make the $\rho_i=\psi_{d_i}$ constructed above a set of basic invariants. They produce a distinguished set of basic invariants $J_1, \ldots J_n$, determined up to non-zero constants, such that $\lambda\in V^*$ gives rise to a set of basic invariants if and only if $J_i(\lambda)\neq 0$ for all $i$ --- see \cite{FlattoWeiner69,Flatto70} for more details.

\begin{theorem3}
\label{exceptional}
Let $W\subset O(V)$ be one of the six exceptional finite reflection groups, and $\rho_1, \ldots \rho_n$ the set of basic invariants described above. Let $ T\subset \{\rho_i\} \cup \{\rho_i\rho_j\}$ be a subset with $n(n+1)/2$ elements such that $T$ contains $\{\rho_i\}$ and
$$\sum_{Q\in T} t^{\deg(Q)-2} =  \frac{P_t(\R[V,\mathrm{Sym}^2V^*]^W)}{P_t(\R[V]^W)}$$

There is at least one such $T$, and for each one,
$ \{ \mathrm{Hess}(Q)\ |\ Q\in T\}$ is a basis for $\R[V,\mathrm{Sym}^2V^*]^W$ as a free module over $\R[V]^W$.
\end{theorem3}

\begin{proof}

First the program finds a list of all subsets $T$ satisfying the condition in the statement of the theorem. The number of elements in this list (choices for $T$) are:

$$\begin{array}{l|l|l|l|l|l|l}
& H_3 & H_4 &  F_4  & E_6 & E_7 & E_8 \\
\hline
\text{choices }& 2 & 2 & 2 & 12 & 48 & 96\\
\end{array}$$

For each $T$, the program constructs a square matrix $M$ of size $n(n+1)/2$. The rows are in correspondence with the set $\mathcal{H}= \{ \mathrm{Hess}(Q)\ |\ Q\in T\}$ , and the columns with the set $\mathcal{P}$ of upper triangular positions of an $n\times n$ matrix. The entry of $M$ associated with $\mathrm{Hess}(Q)\in \mathcal{H}$ and a position $(a,b)\in\mathcal{P}$ is the $(a,b)$-entry of Hess$(Q)(v)$, that is, $$\frac{\partial^2 Q} { \partial r^{\vee}_a \partial r^{\vee}_b} (v)$$

Then it proceeds to compute the determinant of $M$ and checks that it is non-zero. This implies that $\mathcal{H}$ is linearly independent at $v$, hence over $\R[V]$, and in particular over $\R[V]^W$.

Therefore span$_{\R[V]^W}\mathcal{H}$ is a submodule of $\R[V,\mathrm{Sym}^2V^*]^W$ with the same Poincar\'e series, and so they must coincide.
\end{proof}

\bibliographystyle{plain}
\bibliography{ref}

\end{document}